\documentclass{article}
\usepackage{graphicx} % Required for inserting images
\usepackage{tikz-cd}
\usepackage{amsthm}
\usepackage{amsmath}
\usepackage{amssymb}
\usepackage{caption}
\usepackage{txfonts}

\usepackage{natbib}
\newtheorem{theorem}{Theorem}
\newtheorem{lemma}{Lemma}
\newtheorem{proposition}{Proposition}

\theoremstyle{definition}
\newtheorem{mydef}{Definition}[section]

\theoremstyle{remark}

\title{A Logic of Stability: Formalizing Similarity in Counterfactual Reasoning}
\author{Marta Esteves}

\begin{document}

\maketitle
\begin{abstract}
   Counterfactual reasoning is a foundational topic in both philosophical and logical studies \cite{Stalnaker1968-STAATO-5, Lewis1973-LEWC-2}. A pivotal component of counterfactual analysis is the concept of similarity between possible worlds \cite{CORR_A_2022, ESTEVA1997235, Lewis1979-LEWCDA, Makinson94, Pollock1976-POLTPW}. In this paper, we propose the introdutcion of a metric to quantify the degree of similarity between possible worlds, where two worlds are the more similar the longer they share a common history, drawing on a similarity framework influenced by \cite{Lewis1979-LEWCDA}. We prove that this metric satisfies the properties of an ultra-metric, offering a mathematically robust foundation for a corresponding graded notion of hierarchical similarity. We develop and axiomatize a multi-modal logic of similarity, \( L_{\square_\varepsilon} \), and demonstrate its soundness and completeness with respect to the class of ultra-metric spaces. Finally, we explore modal definability, establishing connections between ultra-metric semantics and the broader theory of counterfactuals.
\end{abstract}

\section{Introduction}

This paper introduces a formal account of similarity between possible worlds and the corresponding modal system $L_{\square_\varepsilon}$ (the \emph{logic of stability}). Counterfactuals—sentences of the form ``If $A$ were the case, then $B$ would be the case''—are commonly evaluated by comparing alternative worlds for their similarity to the actual world. Following Lewis, a counterfactual $\phi\boxright\psi$ is true iff all the most similar $\phi$-worlds are $\psi$-worlds. In our approach similarity is measured quantitatively: worlds are interpreted as event-histories and an ultra\-metric on these histories yields a graded, hierarchical notion of nearness. The operators $\square_\varepsilon$ and $\Diamond_\varepsilon$ then express, respectively, that a formula is stable across all worlds within radius $\varepsilon$ and that it is witnessed by some world within radius $\varepsilon$. This makes the link to counterfactuals explicit: stability and plausibility track how truths persist under small, graded perturbations of circumstances, and the ultra\-metric structure implements the ``most similar worlds / closest alternatives'' intuition in a precise, mathematically tractable way.

Because the motivation for axiomatisation of the logic of similarity is tightly linked to the central role this notion plays in counterfactual semantics, we will start by briefly introducing this topic. Counterfactual take the general form, \emph{“If A were the case, then B would be the case”}. As an example, consider the statement: \emph{“If kangaroos didn’t have tails, they would topple over,”}. Counterfactuals play a critical role in philosophical and logical analyses of causation (\cite{Lewis1973-LEWCausation}, decision-making (\cite{Joyce1999-JOYTFO-4}), and scientific explanation (\cite{Woodward2003-WOOMTH}). Foundational contributions by Stalnaker (\cite{Stalnaker1968-STAATO-5}) and Lewis (\cite{Lewis1973-LEWC-2}) have shaped much of the contemporary discussion.

Lewis’ seminal possible-world semantics grounds the truth of a counterfactual in a notion of similarity. Specifically, a counterfactual of the form 
\[
\phi \boxright \psi
\]
is true if, in all the most similar possible worlds where \(\phi\) holds, \(\psi\) also holds. This interpretation hinges on an implicit notion of similarity, often described intuitively as a measure of “closeness” between possible worlds. For example, the counterfactual above can equivalently be read as, \emph{“In all the most similar worlds where kangaroos don’t have tails, they topple over”} in terms of Lewis's possible world semantics.

Lewis illustrates the intuitive notion of similarity through the concept of \emph{similarity spheres}. For a given world \(x\), the spheres represent nested sets of worlds, with smaller spheres containing worlds more similar to \(x\). Consider the diagram below:

\begin{figure}[ht]
    \centering
    \begin{tikzpicture}
        % Draw smaller concentric circles
        \draw[thick] (0,0) circle (1cm);   % inner circle
        \draw[thick] (0,0) circle (2cm);   % outer circle

        % Mark the center point and label it as 'x'
        \filldraw (0,0) circle (1.5pt);
        \node at (0.2,0.2) {\(x\)};

        % Place label x_1 inside the inner circle
        \node at (0.7,0.5) {\(x_1\)};

        % Place label x_2 inside the outer circle
        \node at (1.5,1) {\(x_2\)};
    \end{tikzpicture}
    \caption{Similarity spheres}
    \label{fig:placeholder}
\end{figure}
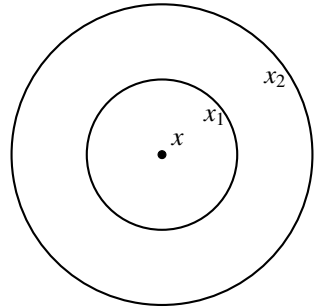

Here, \(x_1\) is closer to \(x\) than \(x_2\), as it lies within a smaller similarity sphere. Our main contribution in this paper is that of building on Lewis's notion of similarity as illustrated by similarity spheres. The novel approach we suggest relies on quantifying similarity by a degree - which ,as we we will prove, behaves as an ultra-metric distance. 

\begin{mydef}
    An ultra-metric on a set \(M\) is a real valued function $d: M \times M \rightarrow \mathbb{R}$, such that, for all $x, y, z \in M$:
\end{mydef}
\begin{itemize}
    \item[1.] $d(x, x) \geq 0$;
    \item[2.] $d(x, y)= d(y, x)$; 
    \item[3.] $d(x, x)=0$; 
    \item[4.] If $d(x, y)=0$ then $y=x$; 
    \item[5.] $d(x, y) \leq max \{d(x, z), d(y, z)\}$ \footnote{Note that an ultra-metric then is a metric that has the triangle inequality in point 5., instead of the usual triangle inequality $d(x, y) \leq d(x, z) + d(y, z)$} 
\end{itemize}
This type of metric provides a structured framework for analyzing the graded similarity between possible worlds. It is important to note, however, that the choice of an ultra-metric is not meant to universally prescribe the notion of similarity between possible worlds. Rather, its relevance arises from the specific interpretation of the structure of the model employed in this paper. The central idea of this model is to interpret possible worlds as binary sequences $x$ of events - where a $0$ at place $x_i$ means that event $E_i$ didn't happen and a $1$ that it did. The structure thus imposed on the set of worlds is hierarchical - as determined by temporal divergence in event histories. The latter then determines the similarity between worlds: more distant worlds are less similar worlds in that their histories diverged earlier in time. We will show later that an adequate model for this is the Cantor space with the ultra-metric distance between sequences.

As a preliminary illustration of how such a metric could be associated to the similarity between possible worlds, consider the following reformulation of the similarity spheres diagram:

  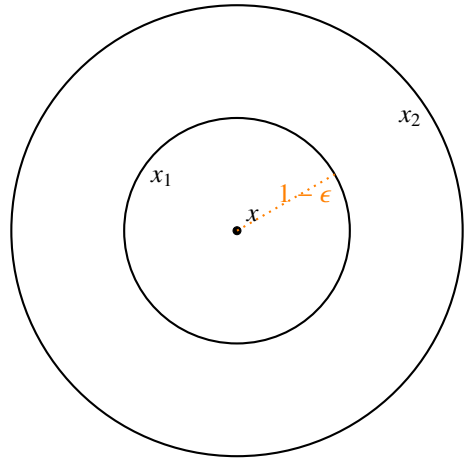
\begin{figure}[ht]
    \centering
    \begin{tikzpicture}
        % Draw concentric circles (larger inner circle)
        \draw[thick] (0,0) circle (1.5cm);   % inner circle
        \draw[thick] (0,0) circle (3cm);     % outer circle

        % Mark the center point and label it as 'x'
        \filldraw (0,0) circle (1.5pt);
        \node at (0.2,0.2) {\(x\)};

        % Place label x_1 inside the inner circle
        \node at (-1,0.7) {\(x_1\)};

        % Place label x_2 inside the outer circle
        \node at (2.3,1.5) {\(x_2\)};

        % Draw a dotted radius for the inner circle in orange, labeled 1 - epsilon
        \draw[dotted, thick, orange] (0,0) -- (1.3,0.75); % Dotted orange line
        \node[orange] at (0.9,0.5) {\(1 - \epsilon\)};
    \end{tikzpicture}
    \caption{Similarity spheres with radius}
    \label{fig:similarity-spheres}
\end{figure}

Where we now say that \(x_1\) is similar to a degree of at most $\epsilon$ to \(x\) because \(x_1\) lies inside a sphere of radius $1- \epsilon$ around \(x\). 

The axiomatization of the logic $L_\epsilon$ relies on this graded notion of similarity and the derived notions of stability and its dual notion of plausibility. This is why the two following notions are foundational to our semantics: 
\begin{mydef}
Inspired by Lange's (\cite{Lange1999-LANLCS}) notion of ``range of invariance" we introduce the following definitions:
\begin{itemize}
    \item[1] \textbf{Range of (the truth of) $\phi$ with respect to a world $w$}: The set $\{v|\exists \epsilon \in [0,1] ( d(w, v)\leq \epsilon) \text{ and } \phi \text{ is true in } v\}$;
    \item[2] \textbf{Degree of stability of $\phi$ in a world $w$} We say that $\phi$ is stable to a degree of $\epsilon$ in a world $w$ iff for all $v$ such that $d(w, v) \leq \epsilon$ then $\phi$ is true at $v$. 
    \item[3] \textbf{Degree of plausibility of $\phi$ in a world $w$} We say that $\phi $ is plausible to a degree of (up to) $1-\epsilon$ in a world $w$ iff there exists $v$ such that $d(w, v) \leq \epsilon$ and $\phi$ is true at $v$.\footnote{Equivalently: if the truth of $\neg \phi$ is not stable in $w$ to a degree of up to $\epsilon$.}
\end{itemize}
\end{mydef}
The intuition behind this definition is that the truth of a sentence $\phi$ is more stable at a world $w$ the more the ``circunstances" - relative to current circumstance - can change while $\phi$ is still true. 

As an example consider the sentences: $p$ =``I have more than 0 euros" and $q=$ `I have more than 10 euros". Then $p$ is more stable than $q$ because $p$ takes bigger changes in someone's current financial circumstances to become false than $q$. In terms of the event history interpretation of possible worlds one says that the truth of a sentence $\phi$ is more stable relative to a given world $w$ the earlier the worlds where $\phi$ is true diverged from $w$. 

Dually, one considers (the truth of) a sentence to be the more plausible the less the current ``circumstances" would have to change for it to be true. And accordingly one interprets $\phi$ being the more plausible in a world $w$ the closer the worlds are where $\phi$ is true.  
For example the sentence ``I went to at least one King Gizzard and the Lizard Wizard concert" is always more plausible than the sentence ``I went to at least two King Gizzard and the Lizard Wizard concerts".

This multimodal logic includes infinitely many quantifiers \(\square_\epsilon\) and their duals \(\Diamond_\epsilon\), with \(\epsilon\) ranging over a countable subset of the unit interval \([0,1]\). The truth of conditions of  \(\square_\epsilon \phi\) and \(\Diamond_\epsilon \phi\) are - according to our semantics - the following: 
\begin{itemize}
    \item $\square_\epsilon \phi$ is true at a world $w$ iff the truth $\phi$ is stable up to a degree of $\epsilon$ (i.e, $\epsilon$-stable) at $w$; 
    \item $\Diamond_\epsilon \phi$ is true at a world $w$ iff $\phi$ is plausible (i.e, plausibly true) up to a degree of $1-\epsilon$ (i.e, $1-\epsilon$-plausible) at $w$; 
    
\end{itemize}
We will see how the axioms of $L_\epsilon$ align with the main intuitions about the stability and plausibility interpretation of the quantifiers $\square_\epsilon$ and $\Diamond_\epsilon$. Furthermore, we will also see that the axioms of $L_\epsilon$ reflect the symmetry, transitivity and reflexivity of the similarity relation, while also guaranteeing that it obeys the ultra-metric triangle inequality axiom (point 5. in Definition 1.1). This will be crucial for our proof that $L_\epsilon$ is sound and complete with respect to the class of ultra-metric spaces.

Now that we are done with the introductory remarks and the presentation of the main objectives of the first part this two-part work, we will start introducing the necessary technical background in the next section. In the second section, ...

\section{The ultra-metric space semantics of similarity}

The move from Lewis’s idea of similarity spheres (\cite{Lewis1979-LEWcount}) to the proposed ultra-metric semantics enables us to give a degree to a particular notion similarity between possible worlds. Building on ideas from topological semantics (\cite{mlspace}, \cite{Baltag2018ATA}), we develop a model in which worlds are represented as binary sequences of events, and their similarity is captured by an ultra-metric that reflects a hierarchically ordered structure of temporal divergence. This transition bridges the philosophical foundations of counterfactuals with a logic-based semantics grounded in metric and topological notions.

As we will see, the models for the logic of similarity $L_\epsilon$ will be ultra-metric space models. This establishes our proposed semantics for the multi-modal logic $L_\epsilon$ as a kind of \textit{topological semantics}.

Topological semantics (\cite{mlspace}, \cite{bluebook}, see also \cite{Baltag2018ATA}, \cite{Baltag2022JustifiedBK} for topological semantics for modal epistemic logic) connects modal logic with topology by interpreting quantifiers as specific subsets of a \textit{topological space}. 
\begin{mydef}
    A topological space is a pair $(X, \tau)$ where \(X\) is a nonempty set and \(\tau\) is a collection of subsets of \(X\) such that:
\end{mydef}
\begin{itemize}
    \item[1.] $\emptyset$ and \(X\) are in \(\tau\); 
    \item[2.] If $U, V \in \tau$, then $U \cap V \in \tau$; 
    \item[3.] If $U_i \in \tau$ is a collection of sets indexed by a set \(I\), then $\bigcup_{i \in I}U_i \in \tau$
\end{itemize}
With this definition in place, we can now introduce two key concepts in topological semantics:

\begin{mydef}
    Interior and closure of a set: 
\end{mydef}

For a subset \(S\) of a topological space \(X\), we define:
\begin{itemize}
    \item [1.] The interior of \(S\) as the union of all open subsets of \(S\) in \(X\);
    \item [2.] The closure of \(S\) as the intersection of all closed sets containing \(S\).
\end{itemize}

In topological semantics, a valuation is a function \(V\) that assigns to each propositional variable \(p\) a subset \(V(p) \subseteq X\), representing the set of points where \(p\) is true. The valuation of the usual modal quantifiers is then given as follows:
\begin{itemize}
    \item \(\square \phi\) is true at a point \(x \in X\) if and only if \(x\) belongs to the interior of the set \([\phi]\), where \([\phi]\) is the set of points in \(X\) where \(\phi\) is true;
    \item \(\Diamond \phi\) is true at \(x \in X\) if and only if \(x\) belongs to the closure of the set \([\phi]\).
\end{itemize}

With these definitions established, we can now define:

\begin{mydef}
    A topological space model is a triple \((X, \tau, V)\), where \((X, \tau)\) is a topological space and \(V\) is a valuation function \(V: P \rightarrow \mathcal{P}(X)\), with \(P\) being a set of propositional variables.
\end{mydef}
Our goal is to develop a semantics for the logic of stability, \(L_\epsilon\), based on ultra-metric spaces. This approach is more specialized than general topological semantics, as the evaluation of formulas in \(L_\epsilon\) is restricted to this specific class of metric spaces. In the next section, we introduce a key concept in this paper: the ultra-metric space model, which forms the foundation of the ultra-metric space semantics for \(L_\epsilon\).

\subsection{Ultra-metric space models}
First and foremost we introduce the language  $\mathcal{L}_{\epsilon}$, given by the following BNF grammar:
\begin{equation}
    \phi: = p \ | \ \phi \land \psi \ |\  \neg \psi \ |\  \square_{\epsilon} \psi \text{ for every $\epsilon\in G \subseteq [0,1]$}.
\end{equation}
Where $G$ is a enumerable subset of $G$ which includes 0 and 1. Such a language is precisely a multi-modal language, in the style discussed in \cite{ESTEVA1997235} \footnote{Note that our approach is similar yet relevantly different from \cite{ESTEVA1997235}, given that the degrees in $L_{\square_\epsilon}$ are meant to represent - in particular - distances in an ultra-metric space and the soundness and completeness of $L_{\square_\epsilon}$ is proved with respect to the class of ultra-metric spaces.}. 
\begin{mydef}
    Let $(X,d)$ be an  ultra-metric space. Let $V:\mathsf{Prop}\to \mathcal{P}(X)$ be a function; we call this a \textit{valuation}. We refer to the triple $(X,d, V)$ as a \textit{model} \footnote{We can also define the useful notion of an Ultra-metric frame $\mathcal{F} = (W, R_\epsilon)_{\epsilon \in D \subseteq G \subseteq [0,1]}$ such that for a corresponding ultra-metric space $(\mathcal{U}, d)$: $W = \mathcal{U}$ and, for any $w, v \in W$,  $w R_\epsilon v$ iff $d(w, v) \leq \epsilon$} of our language, and define the interpretation of all formulas, at a given point, as follows:
    \begin{enumerate}
        \item For any formula $\phi$, we call $[\phi]_V = \{x \in X | (X, d, V) \Vdash \phi\}$ the \textit{truth-set} of $\phi$;
        \item For every $p\in \mathsf{Prop}$, $(X,d,V),x\Vdash p$ if and only if $x\in [p]_V$;
        \item $(X,d,V),x\Vdash \neg\phi$ if and only if $x \notin [\phi]_V$.
        \item $(X,d,V),x\Vdash \phi\wedge \psi$ if and only if $x \in [\phi]_V \cap [\psi]_V$ 
        \item $(X,d,V),x\Vdash \square_{\epsilon}\psi$ if and only if $x \in [\square_\epsilon \psi] = \{y| d(x, y) \leq \epsilon \Rightarrow y \in [\psi]_V\}$
\end{enumerate}
\end{mydef}
In terms of our semantics of similarity then $\square_\epsilon \phi$ being true at a world $w$ can be interpreted to mean that the truth of $\phi$ is stable to a degree of $\epsilon$ in $w$ iff for all worlds that are at most $\epsilon-$distant/dissimilar from $w$ $\phi$ still holds. This matches our initial intuition that more stable truths resist greater changes in current conditions if we interpret a move to a different world as a change in circumstances, where the more distant the world the greater the changes.

Compare the last clause with the usual topological semantics of modal logic (see \cite{mlspace}, \cite{Baltag2018ATA}, \cite{Baltag2022JustifiedBK}): there, given a topological model $\mathcal{M}$, we normally say that \begin{equation*}
    \mathcal{M},x \Vdash \square \psi
\end{equation*}
holds if and only if $x\in Int([\psi])$. Spelling this out, it means that there exists some open set $U$, such that $x\in U$, and $U$ is entirely contained in the truth set of $\psi$. Our definition strengthens this notion by requiring that not only there exists an open set, but in fact an open set which is ``appropriately sized", in the sense that $\psi$ will only be true in the worlds that are at most $\epsilon$-similar to the actual world $w$. This is what will allow to quantify similarity between worlds. 

As such, the operator $\square_{\epsilon}$ can be read as picking up \textit{all} the worlds that are most $\epsilon$-dissimilar from the relevant world. And the operator $\Diamond_{\epsilon}$ can be defined as the Boolean dual, and read as picking \textit{some} world that is at least $1-\epsilon$-similar from the relevant world.

In the following subsection, we will introduce a specific instance of an ultra-metric space model — the Cantor ultra-metric space model - and explain why this model is especially well-suited for capturing a particular notion of similarity between possible worlds.
\subsection{The Cantor Ultra-metric Space Model}
In this section, we introduce the Cantor space, which can be thought of as the space of infinite binary sequences, alongside an ultra-metric distance function defined between these sequences. These elements provide the intended model of the set of possible worlds and the degree of similarity between them.

The core idea of this model is to represent each possible world as an infinite binary sequence encoding a history of events. The degree of similarity between worlds is then determined by the extent to which their event histories align. To clarify this approach, we will begin by outlining the necessary definitions.
\begin{mydef}
    \textit{\textbf{Cantor Space: }}The Cantor space may be described as the topological product of countably many copies of the discrete space $\{0,1\}$. 
\end{mydef}
We can conclude that a point in the Cantor space is an infinite sequence of binary digits. 
\begin{mydef}
    \textit{\textbf{Ultra-metric Cantor Space:}} The Ultra-metric Cantor Space is a pair $(\mathcal{C}, d)$, where $\mathcal{C}$ is the Cantor space and $d$ is a distance function $d: \mathcal{C} \times \mathcal{C} \rightarrow \mathbb{R}$ defined as follows: If $x, y \in \mathcal{C}$ - where $x$ and $y$ are then infinite binary sequences - then $d(x, y) = 1/2^n$, where $n$ is the first index where the sequences $x$ and $y$ differ.
\end{mydef}
It is straightforward to check that $d$ is indeed an ultra-metric distance function and that two points are the more distant the more the corresponding infinite binary sequences differ. 

We will now provide the interpretation of the Cantor space model according to our semantics of similarity and stability of truth, which should be compatible with the stability interpretation of the quantifier $\square_\epsilon$. Consider an infinite sequence of events $\{E_1, E_2, ..., E_n, ...\}$ indexed by the natural numbers. We can represent a possible world as an infinite binary sequence, where a $0$ at index $i$ indicates that event $E_i$ did not occur, and a $1$ signifies that event $E_i$ did occur. The Cantor space consists of such infinite binary sequences, which can be derived from an infinite binary tree. Using a simple diagram of a finite binary tree as an example, we can illustrate how we generate a set of possible worlds and the corresponding similarity distances in our model, interpreting these worlds as histories of events. 
\begin{center}

% https://tikzcd.yichuanshen.de/#N4Igdg9gJgpgziAXAbVABwnAlgFyxMJZAZgBoAGAXVJADcBDAGwFcYkQQBfU9TXfQigBMpAIzU6TVuy48QGbHgJEALGIkMWbRB269FAoiKo1N0nTgD65Wfv7KUo0kI1TtuuQvuCSz11pk9eT4lHzUXUzdAzxDDFABWP0iAnVtggwdkclJif3MQAHdrNK9Qoidc5Pyi0RLYzJFKyRTCyyE6jJ8yJrN3IuIO71UcvL7LFUGyhJGqsfjJuOQANhnm6sslhcyAdlXe9iLtrZ8ncVn2K1qg0sXsiLX3K3bOCRgoAHN4IlAAMwAnCAAWyQ2RAOAgSCcIEY9AARjBGAAFeqCEB-LDvAAWOBA5x0AFFikF-kCQTRwUgRA92AAdGlgGDvAAEhJsxIBwMQUIpiDU0LhCORnXY6KxOLxIEJVzkJM53IhiES1J0dIZzKlaVllPJCpWyslllqNBh8KRKJFGOxmo52rBCt2+tVjJZhutpN5OqQAA5jQKzcKdKKrRLCc8ZTaPXakABOCVO9VtXH801CoaBy049nupU80Sgk2C83psVJ-YExNZzk5hWiKEF-1ptEZ0tRFX052ht2cvW5qllg3tX0potNkuVpA9mtkZOFgOj4OO9sJsO-CMO3N8-udoezxtBzPh93rmtK-vxl3PSicIA
\begin{figure}[ht]
    \centering
    \begin{tikzcd}
        & & t_0 & {} \arrow[ld, "E_0"'] \arrow[rd, "\neg E_0"] & & & & \\
        & t_1 & {} \arrow[ld, "E_1"'] \arrow[rd, "\neg E_1"] & & {} \arrow[d, "E_1"'] \arrow[rd, "\neg E_1"] & & & \\
t_2 & {} \arrow[ld, "E_2"'] \arrow[d, "\neg E_2"'] & & {} \arrow[ld, "E_2"'] \arrow[d, "\neg E_2"'] & {} \arrow[d, "E_2"'] \arrow[rd, "\neg E_2"'] & {} \arrow[rd, "E_2"'] \arrow[rrd, "\neg E_2"] & & \\
w_0 & w_1 & w_2 & w_3 & w_4 & w_5 & w_6 & w_7
    \end{tikzcd}
    \caption{}
    \label{fig:event-tree}
\end{figure}

\end{center}
Where each level of the tree corresponds to a time $t_i$, at which an event $E_i$ happens - represented by a branching to the left - or doesn't happen - in which case $\neg E_i$ happens and we have a branching to the right. 

We can represent the resulting finite space of binary sequences by $\mathcal{U}$, where:

\begin{itemize}
    \item $w_0 = (1,1,1)$, $w_1 = (1,1,0)$, $w_2= (1, 0, 1)$, $w_3=(1, 0, 0)$, $w_4=(0, 1, 1)$, $w_5 = (0, 1, 0)$; $w_6= (0, 0, 1)$, $w_7 = (0, 0, 0)$;
    \item $d(w_0, w_1)=d(w_2, w_3) =d(w_4, w_5) = d(w_6, w_7) = 1/2^3$, $d(w_0, w_2) = d(w_0, w_3)=1/2^2$, $d(w_0, w_4)=d(w_0, w_5)=d(w_0, w_6) = d(w_0, w_7) = 1/2$, and so on.
\end{itemize}

Based on the ultra-metric space model $(\mathcal{U}, d)$ - where $\mathcal{U} = \{w_1, w_2, w_3, w_4, w_5, w_6\}$ and $d$ is the distance between defined as explained, we can define the Ultra-metric space model $\mathcal{M}=(\mathcal{U}, d, V)$, by introducing a valuation $V:\{p, q\} \rightarrow \mathcal{P}(X)$, where \(p\) and $q$ are propositional variables, such that: 
\begin{itemize}
    \item $[p]_v = \{w_0, w_1\}$; 
    \item $[q]_v = \{w_0, w_1, w_2, w_3\}$
\end{itemize}
Then we can see that the following holds: 
\begin{itemize}
    \item $M, w_0 \Vdash \square_{0.125}  p$ because for all \(w\) such that $d(w_0, w) \leq 0.125$ we have that $M, w \Vdash p$;
    \item $M, w_0 \Vdash \square_{0.25} q$ because for all \(w\) such that $d(w_0, w) \leq 0.25$ we have that $M, w \Vdash q$;
\end{itemize}

In order to have a more intuitive picture of these models take \(p\) to denote the proposition ``It rains" and \(q\) to denote the proposition ``there is a thunderstorm". Then at the world $w_0$ the truth of the proposition ``there is a thunderstorm" is more stable than the truth of the proposition ``it rains". This can be interpreted as conveying the fact that, the circumstances in the world $w_0$ could change more and the proposition ``there is a thunderstorm" still be true, than they could change while still preserving the truth of the proposition``it rains".

It is clear that the example just provided is merely a toy example. In fact, it is meant to provide a simple illustration of the infinite case which we will call the Cantor ultra-metric space model. This model is a particular case of an ultra-metric space model and - as we will see -  remarkably well suited suited to our semantics of similarity and stability. Later we will see that this model is not modally definable, and hence it is indistinguishable from any other ultra-metric space model in terms of the language $\mathcal{L}_\epsilon$. 

\begin{mydef}
    \textbf{Cantor ultra-metric space model}: An ultra-metric space model $\mathsf{C} = ((\mathcal{C}, d), V)$ where $(\mathcal{C}, d)$ is the Cantor space with the ultra-metric space distance $d$ defined between sequences, and $V$ is a valuation function $V:\mathsf{Prop}\to \mathcal{P}(\mathcal{C})$. The interpretation of formulas $\phi \in \mathcal{L}_\epsilon$ is the same as in Definition 2.4.
\end{mydef}
It is straightforward to see that $\mathsf{C}$ is an ultra-metric space model. It is in fact a particular case of an ultra-metric space model, with the following features and corresponding semantic interpretation: 
\begin{itemize}
    \item $\mathcal{C}$ is the set of infinite binary sequences. In our semantics this corresponds to possible worlds being interpreted as infinite sequences where a $0$ at index $i$ means that event $E_i$ didn't happen at time $t_i$ and a 1 means that it did;
    \item $d(x, y) = 1/2^n$ where $x$ and $y$ are sequences in $\mathcal{C}$ and $n$ is the first index in which they differ. In our semantics $n$ determines the first point in time $t_n$ where $x$ and $y$ diverged as histories of events. The greater the $n$ the earlier they diverged and the more distant they are. 
\end{itemize}

Note that the proposed models introduce a hierarchical structure to the set of possible worlds, where the distance between two worlds reflects how early their event histories diverged. A way to visualize this is to consider the infinite binary tree generalizing the illustrative example in Fig.?. This allows us to present the more refined semantics of similarity with respect to the Cantor ultra-metric space model:
\begin{itemize}
    \item Two worlds $w$ and $v$ are the more similar the earlier they diverged in terms of their event histories;
    \item A proposition is more stable in a given world $w$ the earlier the worlds diverged in terms of their event histories from $w$ while $\phi$ is still true. 
\end{itemize}

We are now in a position to interpret the intuitive idea of a ``change in circumstances" in more precise terms. I.e, we say that two worlds are more different in terms of their circumstances or conditions the earlier they diverged in terms of their event histories.

An important feature of ultra-metric space models in general and of the cantor ultra-metric space model in particular is that the hierarchical organization is layered, meaning that from the perspective of a particular world, each layer consists of worlds that are all at the same distance from it. Each layer represents a degree of similarity, with the more distant layers corresponding to worlds that diverged earlier in time. This framework, based on ultra-metric spaces, allows for a formalized way of reasoning about how the similarity of one world relative to another depends on their historical proximity.

Here we should note that the structure of ultra-metric spaces imposes a certain universality of the stability of truth that the structure of metric spaces in general doesn't impose. To see what we mean by universality, we look again at the example above and notice that, in all the worlds where either \(p\) or \(q\) is true it is stably true to the same degrees. This is due to the following characteristic of ultra-metric spaces: 
\begin{proposition}
    Every point in a ball in an ultra-metric space is its center.
\end{proposition}
\begin{proof}
    This follows from the fact that for any points $w, u, v$ in an ultra-metric space $d(w, u) \leq max\{d(w, v), d(v, u)\}$ (see \cite{Schikhof_1985}).
\end{proof}

Note that, from the semantics of the quantifier $\square_\epsilon$ and its truth conditions in the ultra-metric space model, one can conclude that the degree of stability of a proposition is equivalently given by the radius of the largest ball around that world in which the proposition holds. Moreover, since every point in an ultra-metric space is the center of a ball, all worlds within this ball are equidistant from one another. Consequently, the proposition maintains the same degree of stability across all these worlds.

Another consequence of the particular structure of ultra-metric spaces on the logic is that, whenever both \(p\) and \(q\) are true in a certain ``region" of the ultra-metric space, it follows that in this region the implication $p \rightarrow q$ is also true. This follows from the following property of ultra-metric spaces: 
\begin{proposition}
    If $\mathcal{U}$ and $\mathcal{V}$ are subsets of an ultra-metric space, then if $\mathcal{U} \cap \mathcal{V} \neq \emptyset$ it follows that either $\mathcal{U} \subseteq \mathcal{V}$ or $\mathcal{V} \subseteq \mathcal{U}$. 
\end{proposition}
\begin{proof}
    This again follows from the ultra-metric triangle inequality condition (see \cite{Schikhof_1985}).
\end{proof}

To see why the above is the case, suppose that for a given ultra-metric space model $M^{\mathcal{U}}= (\mathcal{U}, v)$, we have that $M^{\mathcal{U}}, w \Vdash p \land q$. Then $w \in [p]_v \cap [q]_v$, and therefore either $[p]_v \subseteq [q]_v$ or $[q]_v \subseteq [p]_v$. In the first case we have $M^{\mathcal{U}}, w \Vdash p \rightarrow q$ and in the second $M^{\mathcal{U}}, w \Vdash q \rightarrow p$. 

This universality of stability imposed by the structure of the ultra-metric space model can be explained in light of the interpretation of possible worlds as histories of events. 
Two other topological properties of ultra-metric spaces that are of fundamental importance for our semantics are the ones stated in the following proposition.
\begin{proposition}
    \begin{itemize}
    \item Every Ultra metric space is zero-dimensional;
    \item Every Ultra metric space is Hausdorff.
\end{itemize}
\end{proposition}
\begin{proof}
    Note that every metric space is Hausdorff. The fact that any ultra-metric space has a basis of clopens follows from the fact that every open ball is clopen, as implied by Proposition 1. See \cite{Schikhof_1985} for details.
\end{proof}

The interpretation of these properties in the context of our similarity semantics is given in the context of the Cantor space with the ultra metric distance. In fact, it is well known that the Cantor space is Compact, Hausdorff and zero-dimensional. In order to understand clearly how these properties are interpreted in our semantics it is essential to understand first how the basic opens, opens and closed sets in the Cantor space look like. 
\begin{itemize}
    \item Basic opens: Given that the Cantor space is homeomorphic to $\{0,1\}^{\mathbb{N}}$ with the product topology, the basic opens are cylinder sets. These are in particular sets of sequences that share a finite prefix;
    \item Open sets are - as is well known - unions of basic open sets. In the context of the Cantor space these are sets of sequences that share \textit{some} of many finite prefixes. 
    \item Closed sets are complements of open sets so, in this context, sets of sequences that don't share finite prefixes from a given set of finite prefixes. 
\end{itemize}
Remember that in our semantics the possible worlds, are the elements of the Cantor space, i.e, the binary sequences which we interpreted in terms of histories of events. 
The Cantor ultra-metric space is not only compact, zero-dimensional and Hausdorff, but also perfect (it has no isolated points). This topological feature is important for the intended semantics. Perfection guarantees that for every world $w$ and every radius $\epsilon > 0$ there exist other worlds within distance $\epsilon$ of $w$. Consequently, the modal operators $\square_\epsilon$ and $\Diamond_\epsilon$ quantify over genuinely nontrivial neighborhoods rather than collapsing to singleton tests about the actual world. Philosophically, this matches the idea that circumstances can always be varied slightly; technically, it prevents degenerate validity patterns that occur in ultra-metric models with isolated points (where, for some $\epsilon$, $\Diamond_\epsilon$ may simply restate propositional truth).

\section{The logic of stability $L_\epsilon$}
\subsection{Syntax and Semantics}

First and foremost we introduce the language  $\mathcal{L}_{\square_\epsilon}$, given by the following BNF grammar:
\begin{equation}
    \phi: = p \ | \ \phi \land \psi \ |\  \neg \psi \ |\  \square_{\varepsilon} \psi \text{ for every $\varepsilon\in G \subseteq [0,1]$}.
\end{equation}
Where $G$ is a enumerable subset of $G$ which includes 0 and 1. Such a language is precisely a multi-modal language, in the style discussed in \cite{ESTEVA1997235} \footnote{Note that our approach is similar yet relevantly different from \cite{ESTEVA1997235}, given that the degrees in $L_{\square_\epsilon}$ are meant to represent - in particular - distances in an ultra-metric space and the soundness and completeness of $L_{\square_\epsilon}$ is proved with respect to the class of ultra-metric spaces.}. 

Our semantics of it will involve the familiar concept of a \textit{metric space}.

\begin{mydef}
Let $(X,d)$ be a set equipped with a function $d:X^{2}\to \mathbb{R}$. We say that $d$ is a \textit{metric}, and $(X,d)$ is a \textit{metric space}, if $d$ satisfies the following properties:
\begin{enumerate}
    \item (Positivity and Equality) For each $x,y\in X$, $d(x,y)=0$ if and only if $x=y$.
    \item (Symmetry) For all $x,y\in X$, $d(x,y)=d(y,x)$.
    \item (Triangle Inequality) For all $x,y,z\in X$, $d(x,z)\leq d(x,y)+d(y,z)$.
\end{enumerate}
We say that $d$ is an \textit{ultrametric}, and $(X,d)$ is an \textit{ultrametric space} if $d$ satisfies the following variant of the triangle inequality:
\begin{enumerate}
    \item[3'.] For all $x,y,z\in X$, $d(x,z)\leq \max(d(x,y),d(y,z))$. 
\end{enumerate}
\end{mydef}

\begin{mydef}
    Let $(X,d)$ be an  ultra-metric space. Let $v:\mathsf{Prop}\to \mathcal{P}(X)$ be a function; we call this a \textit{valuation}. We refer to the triple $(X,d, V)$ as a \textit{model} \footnote{We can also define the useful notion of an Ultra-metric frame $\mathcal{F} = (W, R_\epsilon)_{\epsilon \in D \subseteq G \subseteq [0,1]}$ such that for a corresponding ultra-metric space $(\mathcal{U}, d)$: $W = \mathcal{U}$ and, for any $w, v \in W$,  $w R_\epsilon v$ iff $d(w, v) \leq \epsilon$} of our language, and define the interpretation of all formulas, at a given point, as follows:
    \begin{enumerate}
        \item For every $p\in \mathsf{Prop}$, $(X,d,V),x\Vdash p$ if and only if $x\in V(p)$;
        \item $(X,d,V),x\Vdash \neg\phi$ if and only if $(X,d,V),x\nVdash \phi$.
        \item $(X,d,V),x\Vdash \phi\wedge \psi$ if and only if $(X,d,V),x\Vdash \phi$ and $(X,d,V),x\Vdash \psi$.
        \item $(X,d,V),x\Vdash \Box_{\varepsilon}\psi$ if and only if for each $y$ such that $d(x,y)\leq \varepsilon$, then $(X,d,V),y\Vdash \psi$.
\end{enumerate}
\end{mydef}
These interpretations can be restated in terms of truth sets as follows: 
\begin{itemize}
    \item $[\phi] = \{ x \in X \mid (X, d, V), x \Vdash \phi \}$;
    \item $[\neg \phi] = X \setminus\{ x \in X \mid (X, d, V), x \Vdash \phi \}$;
    \item $[\phi \land \psi] = [\phi] \cap [\psi]$;
    \item $[\square_\epsilon \phi] =  \{x \in X: \forall y(d(x, y) \leq \epsilon \Rightarrow y \in A)\}$
\end{itemize}

Compare the last clause with the usual topological semantics of modal logic (see \cite{mlspace}, \cite{Baltag2018ATA}, \cite{Baltag2022JustifiedBK}): there, given a topological model $\mathcal{M}$, we normally say that \begin{equation*}
    \mathcal{M},x \Vdash \square \psi
\end{equation*}
holds if and only if $x\in Int([\psi])$. Spelling this out, it means that there exists some open set $U$, such that $x\in U$, and $U$ is entirely contained in the truth set of $\psi$. Our definition strengthens this notion by requiring that not only there exists an open set, but in fact an open set which is ``appropriately small", in the sense that $\psi$ will only be true in the worlds that are at most $\epsilon-dissimilar$ - or equivalently, $1-\epsilon$ similar - to the actual world $w$.

As such, the operator $\square_{\varepsilon}$ can be read as picking up \textit{all} the worlds that are most $\epsilon-dissimilar$ from the relevant world. And the operator $\Diamond_{\epsilon}$ can be defined as the Boolean dual, and read as picking \textit{some} world that is at most $\epsilon-dissimilar$ from the relevant world. 

\subsection{Axiomatisation}

In this section we present an axiomatisation for a logic $L_\epsilon$ which is intended to capture ultrametric spaces as above. Our axiomatisation mirrors that of \cite{ESTEVA1997235}, given our overall goal is similar. The relevant difference between our axiomatization and theirs is our use of ultra-metric space semantics, which will also figure in the soundness and completeness proofs.

We introduce the following two subsets of the ultra-metric space $(X, d)$, for any arbitrary $A \subseteq X$: 
\begin{itemize}
    \item $I_\varepsilon(A) = \{x \in X : \forall y (d(x, y) \leq \epsilon \Rightarrow y \in A)\}$
    \item $C_\varepsilon(A) = \{x \in X : \exists y (d(x, y) \leq \epsilon \land y \in A)\}$
\end{itemize}
So it is clear that: 
\begin{itemize}
    \item $[\square_\epsilon \phi] = I_\epsilon([\phi])$
    \item $[\Diamond_\epsilon \phi] = C_\epsilon([\phi])$
\end{itemize}

\begin{proposition}
$C_\epsilon$ and $I_\epsilon$ on the set $U$ are then functions $C_\epsilon: \mathcal{P}(U) \rightarrow \mathcal{P}(U)$ and $I_\epsilon: \mathcal{P}(U) \rightarrow \mathcal{P}(U)$ from the power set of $U$ to itself, which satisfy the following conditions for all the sets $A, B \subseteq U$:
 \begin{itemize}
     \item[(i)] $I_\epsilon(A) \subseteq I_\gamma(A)$ for all $\epsilon \geq \gamma$.
     \item[(ii)]  $I_{\epsilon}I_{\gamma}(A) = I_{max(\epsilon,  \gamma)}$
     \item[(iii)] $I_0(A)=A$
     \item[(iv)] $I_\epsilon(A) \subseteq A$
     \item[(v)] $I_\epsilon(A \cap B)=I_\epsilon(A) \cap I_\epsilon(B)$
     \item[(vi)] $A \subseteq C_\epsilon(A)$ 
    \item[(vii)] $C_\epsilon(A) \subseteq I_\epsilon(C_\epsilon(A))$
     \item[(viii)] $(I_\epsilon(A^C))^C = C_\epsilon(A)$ 
      \item[(ix)] $A \subseteq I_\epsilon(C_\epsilon(A))$
 \end{itemize}
 \end{proposition}
 \begin{proof}
\begin{itemize}
    \item[$(i)$] Suppose $x \in I_\epsilon(A)$, then $\forall y(d(x, y) \leq \epsilon \rightarrow y \in A)$. Equivalently, $B_\epsilon(x) \subseteq A$. Now take arbitrary $\gamma \leq \epsilon$. Because $B_\gamma(x) \subseteq B_\epsilon(x) \subseteq A$, it follows that $x \in I_\gamma(x)$.
    \item[$(ii)$] $I_\epsilon(I_\gamma(A)) = \{x \in X | \forall y(d(x, y) \leq \epsilon \rightarrow y \in I_\gamma(A))\} =  \{x \in X | \forall y(d(x, y) \leq \epsilon \rightarrow y \in \{y \in X| \forall z(d(y, z) \leq \gamma \rightarrow z \in A)\}\} = \{x \in X| \forall z(d(x, z) \leq max\{\epsilon, \gamma\} \rightarrow z \in A)\} = I_{\max\{\epsilon, \gamma\}}(A)$. 
    \item[$(iii)$] $I_0(A) = \{x \in X | \forall y(d(x, y) = 0 \rightarrow y \in A)\} = \{x \in X | \forall y(x = y \rightarrow y \in A)\} = \{x \in X | x \in A\} = A$
    \item[$(iv)$] This follows immediately from the definition of $I_\epsilon(A)$. 
    \item[$(v)$] $I_\epsilon(A \cap B) = \{x \in X : \forall y (d(x, y) \leq \epsilon \Rightarrow y \in A \cap B)\} = \{x \in X : \forall y (d(x, y) \leq \epsilon \Rightarrow y \in A \text{ and } y \in B)\} = \{x \in X : \forall y (d(x, y) \leq \epsilon \Rightarrow y \in A)\} \cap \{x \in X : \forall y (d(x, y) \leq \epsilon \Rightarrow y \in B)\} = I_\epsilon(A) \cap I_\epsilon(B)$.
    \item[$(vi)$]  If $x \in A$, then clearly $d(x, x) \leq \epsilon$ so $x \in C_\epsilon(A)$. 
    \item[$(vii)$]  Suppose $x \in C_\epsilon(A)$. Then $\exists y (d(x, y) \leq \epsilon \land y \in A)$. Now consider some arbitrary $z$ such that $d(x, z) \leq \epsilon$. Then $d(y, z) \leq max\{d(x, y), d(x, z)\}=\epsilon$. And because $y \in A$, then $z \in C_\epsilon(A)$. But then $x \in Int_\epsilon(C_\epsilon(A))$.
    \item[$(viii)$]  $(I_\epsilon(A^C))^C = X \setminus \{x \in X | \forall y(d(x, y) \leq \epsilon \rightarrow y \in X \setminus A)\} = \{x \in X | \exists y(d(x, y) \leq \epsilon \land y \notin \neg A)\} = C_\epsilon(A)$. 
    \item[$(ix)$] Suppose $x \in A$. Then take an arbitrary $y$ such that $d(x, y) \leq \epsilon$. Then clearly there is $z \in A$ such that $d(y, z) \leq \epsilon$ and $z \in A$, hence $y \in C_\epsilon(A)$, and therefore $x \in I_\epsilon(C_\epsilon(A))$.  
\end{itemize}
 
\end{proof}

\subsection{Soundness}
The axioms of the logic $L_{\square_\epsilon}$ are the following:
\begin{itemize}
\item[(K)]$\square_\epsilon(\phi \rightarrow \psi) \rightarrow \square_\epsilon \phi \rightarrow \square_\epsilon \psi$, for all $\epsilon \in [0, 1]$
\item[(T)] $\square_\epsilon \phi \rightarrow \phi$, for all $\epsilon \in [0, 1]$
($UM_1$)$\square_\epsilon \phi \rightarrow \Diamond_\epsilon \phi$, for all $\epsilon \in [0, 1]$
\newline
\newline
Inference rules:
\item[(MP)] From $\phi$ and $\phi \rightarrow \psi$ infer $\psi$;
\item[($UM$-Nec)] From $\phi$ infer $\square_\epsilon \phi$, for all $\epsilon \in [0, 1]$
Ultra-metric space axioms:
\item [(TI)] $\square_\gamma \square_\delta \phi \leftrightarrow \square_{max \{\gamma, \delta\}} \phi$, for all $\gamma, \delta \in [0, 1]$
\item[($UM_2$)] $\Diamond_\epsilon \phi \rightarrow \square_\epsilon \Diamond_\epsilon  \phi$, for all $\epsilon \in [0, 1]$
\item[($UM_3$)]$\square_\gamma \phi \rightarrow \square_\delta \phi$,  for all $\gamma, \delta \in [0, 1]$, with $\gamma \geq \delta$
\item[(D)]$\Diamond_\epsilon \phi \leftrightarrow \neg 
\square_\epsilon \neg \phi$, for all $\epsilon \in [0, 1]$
\item[($UM_4$)] $\phi \rightarrow \square_\epsilon \Diamond_\epsilon \phi$, for all $\epsilon \in [0, 1]$
\end{itemize}
 
\begin{proposition}
$L_{\square_\epsilon, \Diamond_\epsilon}$ is sound with respect to the class of Ultrametric Spaces.
 \end{proposition}
\begin{proof}
\begin{itemize}
    \item[(K)] We must prove that the inclusion $I_\epsilon([\phi]^C \cup [\psi]) \subseteq( I_\epsilon([\phi]))^C \cup I_\epsilon[\psi]$ holds in any ultra-metric space. If $x\in I_\epsilon([\psi])$, then it's clear. Now suppose $x \in I_\epsilon([\phi]^C) = I_\epsilon([\neg \phi])$. Then because, by Proposition 3.$(iv)$ $I_\epsilon([\neg \phi]) \subseteq [\neg \phi]$, it follows that $x \in [\neg \phi]$, and by $(vi)$ it follows that $x \in C_\epsilon([\neg \phi]) = C_\epsilon([\phi]^C) = (I_\epsilon([\phi]))^C$, by $(viii)$. 
    \item[(T)] Follows from Proposition 3. $(iv)$
    \item[$(UM_1)$] Follows from Prop.3 $(iv)$ and $(vi)$;
    \item[(TI)] Follows from Prop 3. $(ii)$;
    \item[$(UM_2)$] Follows from Prop 3. $(vii)$;
    \item[$(UM_3)$] Follows from Prop 3. $(i)$;
    \item[(D)] Follows from Prop 3. $(viii)$ 
    \item[$(UM_4)$] Follows from Prop 3. $(ix)$.
    
\end{itemize}
\end{proof}
\subsection{Completeness}
\begin{proposition}
 The logic of generalized knowledge and partial justification $L_{\Diamond_\epsilon, \square_\epsilon}$ is complete with respect to the class of Ultra-Metric Spaces.
 \end{proposition}
 We prove this by defining a canonical ultra-metric space model (a well known-strategy in topological semantics for modal logic \cite{MLmetric}).
We build the canonical metric space model $\mathcal{M}^L_{\square_\varepsilon}= (X^L_{\square_\varepsilon, }, d^L)$ as follows:
\begin{itemize}
    \item  $X^L_{\square_\varepsilon}$ is the set of all maximally consistent sets of $\mathcal{M}$  (the original ultra-metric space model);
    \item The relation in the canonical Ultra-metric space model is defined as follows: $xM_\delta y$ if and only if, for every formula $\phi$: $\square_\delta \phi \in x \Rightarrow \phi \in y$ \footnote{Note that this is equivalent to saying that $\phi \in y \Rightarrow \diamond_\delta \phi \in x$, see Lemma 4.19 in \cite{bluebook}};
    \item $d^L_{\square_\varepsilon}$ is the distance between points in $X^L$, defined as follows: $d^L_{\square_\varepsilon}(x, y)= inf\{\delta: xM_\delta y\}$
    
\end{itemize}

\begin{lemma} $\mathcal{M}^L_{\square_\varepsilon}= (X^L_{\square_\varepsilon}, d^L, \nu^L)$ is an Ultra-Metric Space:
\end{lemma}
\begin{proof}
    \begin{itemize}
        \item[1.] Symmetry: $d^L_{\square_\epsilon}(x, y) = d^L_{\square_\epsilon}(y, x)$: We need to show that $inf\{\delta: xM_\delta y\} = inf\{\delta: yM_\delta x\}$. Take arbitrary $\alpha \in [0, 1]$ such that $xM_\alpha y$. Then take arbitrary $\phi \in w$. By the $(S_4)$ axiom, it follows that $\square_\alpha\Diamond_\alpha \phi \in x$. From the definition of $xM_\alpha y$ it follows that $\square_\alpha(\Diamond_\alpha\phi) \in x \Rightarrow \Diamond_\alpha \phi \in y$, and by MP: $\Diamond_\alpha \phi \in y$. So for arbitrary $\alpha \in [0,1]$ and $\phi$ we have that $\phi \in x$ entails $\Diamond_\alpha\phi \in y$ and hence that $yM_\alpha x$.

        Following the same reasoning, from $yM_\alpha x$ we can conclude that $xM_\alpha y$. So it follows that $\{\delta: xM_\delta y\} = \{\delta: yM_\delta x\}$ and hence $inf\{\delta: xM_\delta y\} = inf\{\delta: yM_\delta x\}$. 
        \item[2.] Reflexivity: $d^L_{\square_\epsilon}(x, x) = 0$.  By axiom $(T)$ it follows that, for all $\gamma \in [0,1]$, and for all $x\in X^L_{\square_\epsilon}$ $\square_\gamma \phi \rightarrow \phi \in x$, and hence by definition, for all $\gamma \in [0,1]$ we have that $xM_\gamma x$, and therefore $inf\{\delta : xM_\delta x\} = 0$.
        \item[3.] Ultra-metric triangle inequality: Take three worlds $x, y, z \in X^L_{\square_\epsilon}$ such that $d^L_{\square_\epsilon}(x, y) = \gamma$, $d^L_{\square_\epsilon}(y, z) = \delta$. We want to prove that $d^L_{\square_\epsilon}(x, z) \leq max \{\gamma, \delta\}$. So suppose $\square_{max\{\gamma, \delta\}} \phi \in x$, then by axiom (TI) it follows that $\square_\gamma \square_\delta \phi \in x$. By the definition of $d^L_{\square_\epsilon}$ it follows that $xM_\gamma y$ and hence that $\square_\delta \phi \in y$, and because $yM_\delta z$ it follows that $\phi \in z$. But then $xM_{max\{\gamma, \delta\}} z$, and hence $d(x, z) \leq max\{\gamma, \delta\}$. 
    \end{itemize}
\end{proof}
\begin{lemma} Truth Lemma: For all formulas $\phi$ in $L_{\square_\epsilon}$: 
$$\mathcal{M}^L_{\square_\epsilon}, x \Vdash \phi \Leftrightarrow \phi \in x$$
\begin{proof}
    By induction on the complexity of $\phi$. We focus on the non-trivial case for $\Diamond_\epsilon\phi$. 
    \begin{itemize}
        \item[1.] Forward direction ($\Rightarrow$) Assume $\mathcal{M}^L_{\square_\epsilon} \Vdash \Diamond_\epsilon \phi$. Then there exists $y \in X^L_{\square_\epsilon}$ such that $d^L_{\square_\epsilon}(x, y) \leq \epsilon$ and $\mathcal{M}^L_{\square_\epsilon}, y \Vdash \phi$. But then, by induction hypothesis $\phi \in y$. And because then there exists $x$ such that $y M_\epsilon x$, it follows by definition of $M_\epsilon$ that $\diamond_\epsilon\phi \in x$.
        \item[2.] Reverse direction ($\Leftarrow$) Assume $\Diamond_\epsilon \phi \in x$. We need to construct $y \in \mathcal{M}^L_{\square_\epsilon}$ such that $xM_\epsilon y$ and $\phi \in y$. I.e, we need to construct a consistent set: 
        $$\Gamma = \{\Diamond_\epsilon \psi | \psi \in x\} \cup \{\phi\}$$
        and extend it to a maximally consistent set $y$ such that $\phi \in y$ and $d^L_{\square_\epsilon}(x, y) \leq \epsilon$. 

        First, we prove that $\Gamma$ is consistent. For suppose it is not, then: 
        $$\vdash (\Diamond_\epsilon \psi_1 \land ...\land \Diamond_\epsilon \psi_n) \rightarrow \neg \phi$$
        By necessitation and the normality axiom, we have:
          $$\vdash (\square_\epsilon\Diamond_\epsilon \psi_1 \land ...\land \square_\epsilon\Diamond_\epsilon \psi_n) \rightarrow \square_\epsilon\neg \phi$$
    \end{itemize}
    But then $(\square_\epsilon\Diamond_\epsilon \psi_1 \land ...\land \square_\epsilon\Diamond_\epsilon \psi_n) \rightarrow \square_\epsilon\neg \phi \in x$, with $\psi_1, ..., \psi_n \in x$. By axiom $(S_4)$ it follows that $\square_\epsilon\Diamond_\epsilon \psi_1, ..., \square_\epsilon \Diamond_\epsilon\psi_n \in x$, and by MP $\square_\epsilon \neg \phi \in x$. But this is a contradiction, so $\Gamma$ is consistent. 
    We can then extend $\Gamma$ to a maximally consistent set $y$ using Lindenbaum's Lemma (\cite{bluebook}). By construction $\phi \in y$ so by I.H, $\mathcal{M}^L_{\square_\epsilon}, y \Vdash \phi$. Also by construction $\psi \in x \Rightarrow \Diamond_\epsilon \psi \in y$, so $yM_\epsilon x$. Because the relation $M_\epsilon$ was proved to be symmetric it follows that: $d^L_{\square_\epsilon}(x, y) = inf \{\delta | x M_\delta y\} \leq \epsilon$. Therefore $\mathcal{M}^L_{\square_\epsilon}, x \Vdash \Diamond_\epsilon \psi$. 
\end{proof}
    
\end{lemma}
\section{Definability results}
In this section we present some definability results for the logic $L_{\square_\epsilon}$. Most importantly, we prove that the Cantor Space is not definable in the language $\mathcal{L}_{\square_\epsilon}$. This is relevant because - as we previously elaborated - the Cantor Space is the most natural model for the semantics of stability of truth.

We also prove that the class of perfect ultra-metric spaces is not definable in $\mathcal{L}_{\square_\varepsilon}$. 
As we have clarified earlier, this property of the Cantor space is essential to guarantee the \emph{non-degeneracy of the graded modalities}: for every world $w$ and every $\varepsilon>0$ there are other worlds within distance $\varepsilon$ of $w$, so $\square_\varepsilon$ and $\Diamond_\varepsilon$ quantify over genuinely nontrivial neighbourhoods rather than collapsing to propositional truth. 

Consequently, although $\mathcal{L}_{\square_\varepsilon}$ cannot syntactically single out perfection, perfection remains a semantically important extra assumption for the intended Cantor semantics.

We now proceed to present some results concerning the preservation of validity of formulas in the language $\mathcal{L}_\epsilon$ with respect to some constructions on ultra-metric spaces. 
\begin{mydef}
    Modal satisfaction with respect to an ultra-metric space model $(\mathcal{U}, d)$: We say that a formula $\phi \in \mathcal{L}_{\square_\epsilon}$ is satisfied in an ultra-metric space model $((\mathcal{U}, d), V)$ iff there is a valuation $V$ and a world $w \in \mathcal{U}$ such that: $(\mathcal{U}, d), V, w \Vdash \phi$. 
\end{mydef}
\begin{mydef}
    Modal validity with respect to an ultra-metric space $(\mathcal{U}, d)$: We say that a formula $\phi \in \mathcal{L}_{\square_\epsilon}$ is valid in an ultra-metric space $(\mathcal{U}, d)$ -- written as $(\mathcal{U}, d) \Vdash \phi$ -- iff for any valuation $V$ and for any world $w \in \mathcal{U}$: $(\mathcal{U}, d), V, w \Vdash \phi$. 
\end{mydef}
\begin{mydef}
    Ultra-metric bounded morphism (models): An ultra-metric bounded morphism between ultra-metric space models $\mathcal{M} = (W, d, V)$ and $\mathcal{M'} = (W', d', V')$ is a function $f: W \rightarrow W'$ and $k> 0$ such that: 
\end{mydef}
\begin{itemize}
    \item[1.] For all atomic propositions $p$: $w \in V(p) \Leftrightarrow f(w) \in V'(p)$ for all propositions $p$;
    \item[2.] Forward condition: for all $\epsilon \in [0, 1]$: $d(w, v) \leq \epsilon$ implies $d(f(w), f(v)) \leq  k.\epsilon$; 
    \item[3.] Back condition: If $d'(f(w), v') \leq \epsilon$ then there exists $v$ such that,  $d(w, v) \leq (k^{-1}).\epsilon$ and $f(v) = v'$.  
\end{itemize}
\begin{mydef}
    Ultra-metric bounded morphism (frames): An ultra-metric bounded morphism from an ultra-metric space $(\mathcal{U}, d)$ to another $(\mathcal{U}, d)$ is a function $f: W \rightarrow W'$ such that, for some $k > 0$: 
\end{mydef}
\begin{itemize}
     \item[1.] Forward condition: for all $\epsilon \in [0, 1]$: $d(w, v) \leq \epsilon$ implies $d(f(w), f(v)) \leq k.\epsilon$; 
    \item[2.] Back condition: If $d'(f(w), v') \leq \epsilon$ then there exists $v$ such that $d(w, v) \leq (k^{-1}).\epsilon$ and $f(v) = v'$.  
\end{itemize}
\begin{mydef}
    The disjoint union of a family of ultra-metric space models$\{\mathcal{M}_i = (W_i, d_i, V_i)\}_{i \in I}$ is $\biguplus \mathcal{M_i} = (W, d, V)$, where:
\end{mydef}
\begin{itemize}
    \item $W = \biguplus_{i \in I} W_i$;
    \item 
    \[
    d(x, y) = 
    \begin{cases}
        d_i(x, y) & \text{if } x, y \in W_i, \\
        2 & \text{otherwise.}
    \end{cases}
    \]
    \item $V(p) = \biguplus_{i \in I} V_i(p)$
\end{itemize}

It is straightforward to verify that $(W, d)$ resulting from this construction is indeed an ultra-metric space. 
\begin{mydef}
$\epsilon$-generated subspace: Let $(\mathcal{U}, d)$ be an ultra-metric space. Then we say that $(B_\epsilon(x), d')$ is the $\epsilon$-generated subspace of $(\mathcal{U}, d)$ (by a point $x\in \mathcal{U}$ ) if:
\begin{itemize}
    \item $B_\epsilon(x) \subseteq \mathcal{U}$, as $B_\epsilon(x) = \{y \in \mathcal{U}| d(x, y) \leq \epsilon \land x \in \mathcal{U}\}$
    \item $d'(x, y) = d(x, y)$ for all $x, y \in \mathcal{U}$. 
\end{itemize}
\end{mydef}

\begin{proposition}
    For $\{\mathcal{M}_i =(\mathcal{U}_i, d_i)| i \in I\}$ a family of ultra-metric space  models, $\biguplus_{i \in I}\mathcal{M}_i$ their disjoint union, and for all $\phi \in \mathcal{L}_{\square_\epsilon}$:
    $$\mathcal{M}_i, w \Vdash \phi \Leftrightarrow \biguplus_{i \in I} \mathcal{M}_i, w \Vdash \phi$$
    I.e,  modal satisfaction is invariant under disjoint unions of ultra-metric space models.
\end{proposition}
\begin{proof}
    By induction. We leave out the boolean cases and consider only that of the quantifier  $\Diamond_\epsilon$:
\begin{itemize}
    \item If $\mathcal{M}_i, w \Vdash \Diamond_\epsilon \phi$ , then there is $v$ with $d_i(w, v) \leq \epsilon$ and $\mathcal{M}_i, v \Vdash \phi$. By induction hypothesis $\biguplus_{i \in I}\mathcal{M}_i, v \Vdash \phi$, and because $d(w, v) = d_i(w, v)$ it follows that $\biguplus_{i \in I}\mathcal{M}_i, v \Vdash \Diamond_\epsilon\phi$. 
    \item If $\biguplus_{i \in I}\mathcal{M}_i, w \Vdash \Diamond_\epsilon\phi$ for some $w \in \mathcal{M}_i$, then there exists $v \in \biguplus_{i \in I}\mathcal{M}_i$ such that $d(w, v) \leq \epsilon$ and $\biguplus_{i \in I}\mathcal{M}_i, v \Vdash \phi$. By definition $d(w, v) = d_j(w, v)$ for some $j \in J$ and because the union is disjoint $j = i$. But then by I.H: $\mathcal{M}_i, v \Vdash \phi$ and because $d_i(w, v) \leq \epsilon$, $\mathcal{M}_i, v \Vdash \Diamond_\epsilon\phi$. 
\end{itemize}
\end{proof}
\begin{proposition}
Consider two ultra-metric space models $\mathcal{M}=((\mathcal{U}, d), V)$ and $\mathcal{M}'=((\mathcal{U'}, d'), V)$. Then if $f$ is an ultra-metric bounded morphism between $\mathcal{M}$ and $\mathcal{M'}$:
\begin{itemize}
    \item for all $\phi \in \mathcal{L}_{\square_\epsilon}$, such that $\phi$ is not a quantified formula:
$$\mathcal{M}, w \Vdash \phi \Leftrightarrow \mathcal{M'}, f(w) \Vdash \phi$$
\item  for all $\phi \in \mathcal{L}_{\square_\epsilon}$, such that $\phi$ is a quantified formula and $k >0$ :
$$\mathcal{M}, w \Vdash \Diamond_\epsilon \phi \Leftrightarrow \mathcal{M'}, f(w) \Vdash \Diamond_{k.\epsilon} \phi$$
\end{itemize}

\end{proposition}
\begin{proof}
    We skip the propositional variables and the boolean cases and consider the case of the quantifier $\Diamond_\epsilon \phi$. 
    \begin{itemize}
        \item Suppose $\mathcal{M}, w \Vdash \Diamond_\epsilon \phi$. Then there exists $v \in \mathcal{U}$ such that $d(w, v) \leq \epsilon$ and $\mathcal{M}, v \Vdash \phi$. By I.H  $\mathcal{M}', f(v) \Vdash \phi$. By the forwards condition then, for $k > 0$: $d(f(w), f(v)) \leq k.\epsilon$, and hence $\mathcal{M}', f(w) \Vdash \Diamond_{k.\epsilon}\phi$. 
\item  Suppose $\mathcal{M}', f(w) \Vdash \Diamond_{k.\epsilon} \phi$. Then $\exists v' \in \mathcal{U}$ such that $d'(f(w), v') \leq k.\epsilon$ and $\mathcal{M}', v' \Vdash \phi$. By the backwards condition there is a $v \in \mathcal{U}$ such that there is $k'>0$ (namely $k' = k^{-1}$), $d(w, v) \leq k^{-1}.k.\epsilon = \epsilon$ and  $f(v) = v'$. If $\phi$ is a quantified formula we just repeat this n-many times for quantifier depth $n$. If it isn't, then by I.H: $\mathcal{M}, v \Vdash \phi$. Therefore $\mathcal{M}, w \Vdash \Diamond_{\epsilon} \phi $
   \end{itemize}
\end{proof}
\begin{theorem}
    Let $\phi$ be any formula in the language $\mathcal{L}_{\square_\epsilon}$:
    \begin{itemize}
        \item[1.]Let $\{U_i =(\mathcal{U}_i, d_i)| i \in I\}$ be a family of ultra-metric spaces. Then $\biguplus_{i \in I} U_i \Vdash \phi$ if $U_i \Vdash \phi$ for every $i \in I$.
        \item[2.] Assume $U' = (B_\epsilon(x), d')$ is an $\epsilon$-generated subspace of $U = (\mathcal{U}, d)$. Then $U' \Vdash \phi$ if $U \Vdash \phi$. 
        \item[3.] Assume there is a surjective bounded morphism from $U = (\mathcal{U}, d)$ onto $U' = (\mathcal{U}', d')$. Then $U' \Vdash \phi$ if $U \Vdash \phi$ if $\phi$ is a non quantified formula. Otherwise, if $U \Vdash \Diamond_\epsilon\phi$, then there is $k >0$ such that $U' \Vdash \Diamond_{k.\epsilon} \phi$.
    \end{itemize}
    \begin{proof}
        \begin{itemize}
            \item[(i)] Suppose, for some arbitrary $i \in I$ $U_i\Vdash \phi$. We want to show that $\biguplus_{i \in I} U_i \Vdash \phi$. We define: 
            $$V(p_i) = \biguplus_{i \in I} V_i(p_i)$$
            Because $\phi$ is valid in $U_i$ this means that, for an arbitrary valuation $V_i$ (as defined in Definition 4.5) over $U_i$ and for any $w \in \mathcal{U}_i$: $U_i, V_i, w \Vdash \phi$. By Proposition 7. $\biguplus_{i \in I}U_i, V, w \Vdash \phi$. Hence $\biguplus_{i \in I}U_i\Vdash \phi$
            \item[(ii)] First we define the valuation function $V': P \rightarrow B_\epsilon(x)$ such that, for any propositional variable $p$: $V'(p) =V(p) \cap B_\epsilon(x)$, for $V$ any valuation function $V: P \rightarrow \mathcal{U}$. Suppose that $U' \nVdash \phi$. Then if $U' \nVdash \phi$, there is some $w \in B_\epsilon(x)$  such that: $U', V', w \Vdash \neg \phi$. By a straightforward inductive argument one can conclude that $V'(\neg \phi) = V(\neg \phi) \cap B_\epsilon(x)$ for any formula $\phi$ of $L_{\square_\epsilon}$.
            follows that: 
            $U, V, w \Vdash \neg \phi$. Therefore $U \nVdash \phi$.
            \item[(iii)] We first define the following valuation, over the ultra-metric space $U'$:
            \begin{equation}
                   V'(p_i) = \{f(x) \in \mathcal{U} | x \in V(p_i) \}
            \end{equation}
            Where $V(p_i)$ is any arbitrary valuation over $U$.
            Then we consider the following two cases:
            \begin{itemize}
                \item $\phi$ is not a quantified formula: Suppose that $U \Vdash \phi$. Then, for an arbitrary valuation $V$ , and for any world $u \in \mathcal{U}$, we have: $U, V, u \Vdash \phi$. From $(1)$  it follows that $f$ is a bounded morphism between $\mathcal{M} = (U, V)$ and $\mathcal{M'} = (U', V')$. Hence, by proposition 8.: $U', V', f(u) \Vdash \phi$. But now note that $f$ is surjective, and therefore $f[\mathcal{U}] = \mathcal{U}'$. And because $u \in \mathcal{U}$ was taken to be arbitrary, and $V$ is arbitrary, and therefore also $V'$: $U' \Vdash \phi$; 
                \item $\phi$ is a quantified formula: Suppose $U \Vdash \Diamond_\epsilon \phi$. Then, for an arbitrary valuation $V$ and for any world $u \in \mathcal{U}$, we have $U, V, u \Vdash \Diamond_{\epsilon} \phi$. Again because $f$ in $(1)$ is a bounded morphism and by Proposition 8., we have: $U', V', f(u) \Vdash \Diamond_{k.\epsilon} \phi$. Because $f$ is surjective and $V$ and hence $V'$ are arbitrary:  $U' \Vdash \Diamond_{k.\epsilon}$.
                
            \end{itemize}
            
        \end{itemize}
    \end{proof}
\end{theorem}

\begin{lemma}
    The Ultra-metric Cantor Space is not modally definable.
\end{lemma}
\begin{proof}
    Let $\mathcal{C} = (C, d)$ be the Ultra-metric Cantor space - where $C$ is the Cantor space and $d$ the ultra-metric distance between the points/ sequences of $C$. Now we take the disjoint union of $\kappa$- many, where $\kappa > 2^{\aleph_0}$, copies of $\mathcal{C}$: $\biguplus_{i \in I}\mathcal{C}_i$ (with $|I| = \kappa$). We have proved that the disjoint union of ultra-metric spaces is validity preserving. So if $\mathcal{C}$ is modally definable, i.e,  there is a formula $\phi$ in the language $\mathcal{L}_{\square_\epsilon}$ defining $\mathcal{C}$, it is valid in $\biguplus_{i \in I}\mathcal{C}_i$. But $\biguplus_{i \in I}\mathcal{C}_i$ has cardinality greater than $2^{\aleph_0}$ so it cannot be the Cantor Space. This is a contradiction, so the Cantor Space is not modally definable. 
\end{proof}
\begin{mydef}
    A \textit{perfect} topological space is a topological space where all the points are limit points. Where limit point of a subset $S$ in a topological space $X$ is a point $x \in X$ such that every open neighborhood of $x$ contains at least one point distinct from $x$. 
\end{mydef}

\begin{lemma}
    The class of perfect ultra-metric spaces is not modally definable in $\mathcal{L}_{\square_\epsilon}$.
\end{lemma}
\begin{proof}
    Suppose there is a formula from $\mathcal{L}_{\square_\epsilon}$ defining the property of an ultra-metric space being perfect. 

    Consider the ultra-metric Cantor space $\mathcal{C}$ again. Its is clear that it is a perfect space. Now consider the $\epsilon$-generated subspace of $\mathcal{C}$: $(B_0(x), d')$ where $d(x, x) =d'(x, x)$.  This is clearly the trivial ultra-metric space with underlying set $\{x\}$. But then $(B_0(x), d')$ is not a perfect set. And because $\epsilon$-generated subspaces preserve the validity of formulas then $\phi$ would be valid in $(B_0(x), d')$. But this is a contradiction. 
\end{proof}

\begin{mydef}
    A function $f: (X, d_X) \rightarrow (Y, d_Y)$ where $(X, d_X)$ and  $(Y, d_Y)$ are metric spaces is a lipschitz continuous function whenever there is $k \in \mathbb{R}, \text{ }k \geq 0$ such that: $d(f(x), f(y))\leq k.d_X(x, y)$. And $f$ is bilipschitz continuous if there is $k \in \mathbb{R}, \text{ }k \geq 1$ such that $k^{-1}d_X(x, y) \leq d_Y(f(x), f(y)) \leq k.d_X(x, y)$. 
\end{mydef}
\begin{proposition}
    Let $f$ be an ultra-metric bounded morphism from an ultra-metric space $(\mathcal{U}, d)$ to an ultra-metric space $(\mathcal{U'}, d')$. Then $f$ is a bilipschitz homeomorphism. In other words $(\mathcal{U}, d)$ and $(\mathcal{U'}, d')$ are bilipschitz homeomorphic.  
\end{proposition}
\begin{proof} Let $f:(\mathcal{U}, d) \rightarrow (\mathcal{U'}, d')$ be an ultra-metric bounded morphism. Then for all $x, y \in \mathcal{U}$: $d'(f(x), f(y)) \leq k.d(x, y)$ follows from the forward condition. And $k^{-1}.d(x, y) \leq d(f(x), f(y))$ follows from the back condition. 
    
\end{proof}

This makes explicit the intuition that an ultra-metric bounded morphism is a map that distorts distances by at most a fixed multiplicative constant. The forward/back conditions with constant $k>0$ amount to a one-sided Lipschitz control $d'(f(x),f(y))\le k\cdot d(x,y)$ together with the fact that preimages of balls are contained in proportionally smaller balls; equivalently, $f(B_\varepsilon(x))\subseteq B_{k\varepsilon}(f(x))$ and $f^{-1}(B_\varepsilon(f(x)))\subseteq B_{k^{-1}\varepsilon}(x)$. If $f$ is bijective these bounds run in both directions and we obtain the bi-Lipschitz inequalities $k^{-1}d(x,y)\le d'(f(x),f(y))\le k\,d(x,y)$, so $f$ is a bi-Lipschitz homeomorphism. Semantically, this explains the preservation results: propositional formulas are preserved exactly under $f$, while quantified formulas are preserved up to the radius rescaling $\varepsilon\mapsto k\varepsilon$ (and dually $\varepsilon\mapsto k^{-1}\varepsilon$ on preimages). When $k=1$ the map is distance-nonexpanding and, if bijective, an isometry.
\section{Conclusion}

We have presented a graded semantics of similarity for possible-world counterfactuals based on ultra\-metric spaces, and developed a matching multi-modal logic \(L_{\square_\varepsilon}\). Concretely, we (i) introduced a natural ultra\-metric on histories (Cantor-style binary sequences) that captures hierarchical divergence of event-histories, (ii) axiomatized the logic of stability and plausibility, and (iii) proved soundness and completeness with respect to the class of ultra\-metric spaces via a canonical ultra\-metric model construction. 

Our definability results show two informative limits of the language: the Cantor space and the class of perfect ultra\-metric spaces are not definable in \(\mathcal{L}_{\square_\varepsilon}\). Semantically this is not a shortcoming but a clarification of scope: \(\mathcal{L}_{\square_\varepsilon}\) captures the local, neighbourhood-based content of graded similarity (the intended object of study), while global topological features such as perfectness must be added as extra semantic assumptions if one wants them guaranteed in every model.

In further work, we will use the similarity apparatus developed here to axiomatize a graded counterfactual logic and to investigate further connections with counterfactual semantics. More generally, the present framework invites several extensions: such as investigating computational and decidability questions for \(L_{\square_\varepsilon}\), and exploring further applications to formal epistemology and causal modeling. We conclude that ultra\-metric semantics provide a robust and conceptually clear setting for graded similarity, yielding a precise account of stability and plausibility that is both technically tractable and philosophically natural.


\newpage

\begin{thebibliography}{18}
\providecommand{\natexlab}[1]{#1}
\providecommand{\url}[1]{\texttt{#1}}
\expandafter\ifx\csname urlstyle\endcsname\relax
  \providecommand{\doi}[1]{doi: #1}\else
  \providecommand{\doi}{doi: \begingroup \urlstyle{rm}\Url}\fi

\bibitem[Baltag et~al.(2018)Baltag, Bezhanishvili, {\"O}zg{\"u}n, and Smets]{Baltag2018ATA}
Alexandru Baltag, Nick Bezhanishvili, Ayb{\"u}ke {\"O}zg{\"u}n, and Sonja Smets.
\newblock A topological approach to full belief.
\newblock \emph{Journal of Philosophical Logic}, 48:\penalty0 205 -- 244, 2018.
\newblock URL \url{https://api.semanticscholar.org/CorpusID:52201580}.

\bibitem[Baltag et~al.(2022)Baltag, Bezhanishvili, {\"O}zg{\"u}n, and Smets]{Baltag2022JustifiedBK}
Alexandru Baltag, Nick Bezhanishvili, Ayb{\"u}ke {\"O}zg{\"u}n, and Sonja Smets.
\newblock Justified belief, knowledge, and the topology of evidence.
\newblock \emph{Synthese}, 200:\penalty0 1--51, 2022.
\newblock URL \url{https://api.semanticscholar.org/CorpusID:254276616}.

\bibitem[Benthem and Bezhanishvili(2007)]{mlspace}
Johan Benthem and Guram Bezhanishvili.
\newblock \emph{Modal Logics of Space}, pages 217--298.
\newblock 09 2007.
\newblock ISBN 978-1-4020-5586-7.
\newblock \doi{$10.1007/978-1-4020-5587-4_5$}.

\bibitem[Bezhanishvili et~al.(2014)Bezhanishvili, Gabelaia, and Lucero-Bryan]{MLmetric}
Guram Bezhanishvili, David Gabelaia, and Joel Lucero-Bryan.
\newblock Modal logics of metric spaces.
\newblock \emph{The Review of Symbolic Logic}, 8:\penalty0 178--191, 03 2014.
\newblock \doi{10.1017/S1755020314000446}.

\bibitem[Corrêa and Oliveira(2022)]{CORR_A_2022}
Nicholas Corrêa and Nythamar Fernandes~de Oliveira.
\newblock Counterfactual analysis by algorithmic complexity: A metric between possible worlds.
\newblock \emph{Manuscrito}, 45\penalty0 (4):\penalty0 1–35, December 2022.
\newblock ISSN 0100-6045.
\newblock \doi{10.1590/0100-6045.2022.v45n4.nn}.
\newblock URL \url{http://dx.doi.org/10.1590/0100-6045.2022.V45N4.NN}.

\bibitem[Esteva et~al.(1997)Esteva, Garcia, Godo, and Rodríguez]{ESTEVA1997235}
Francesc Esteva, Pere Garcia, Lluís Godo, and Ricardo Rodríguez.
\newblock A modal account of similarity-based reasoning.
\newblock \emph{International Journal of Approximate Reasoning}, 16\penalty0 (3):\penalty0 235--260, 1997.
\newblock ISSN 0888-613X.
\newblock \doi{https://doi.org/10.1016/S0888-613X(96)00126-0}.
\newblock URL \url{https://www.sciencedirect.com/science/article/pii/S0888613X96001260}.

\bibitem[Joyce(1999)]{Joyce1999-JOYTFO-4}
James~M. Joyce.
\newblock \emph{The Foundations of Causal Decision Theory}.
\newblock Cambridge University Press, 1999.

\bibitem[Lange(1999)]{Lange1999-LANLCS}
Marc Lange.
\newblock Laws, counterfactuals, stability, and degrees of lawhood.
\newblock \emph{Philosophy of Science}, 66\penalty0 (2):\penalty0 243--267, 1999.
\newblock \doi{10.1086/392686}.

\bibitem[Lewis(1973{\natexlab{a}})]{Lewis1973-LEWCausation}
David Lewis.
\newblock Causation.
\newblock \emph{Journal of Philosophy}, 70\penalty0 (17):\penalty0 556--567, 1973{\natexlab{a}}.
\newblock \doi{10.2307/2025310}.

\bibitem[Lewis(1979{\natexlab{a}})]{Lewis1979-LEWCDA}
David Lewis.
\newblock Counterfactual dependence and time?s arrow.
\newblock \emph{No\^{u}s}, 13\penalty0 (4):\penalty0 455--476, 1979{\natexlab{a}}.
\newblock \doi{10.2307/2215339}.

\bibitem[Lewis(1979{\natexlab{b}})]{Lewis1979-LEWcount}
David Lewis.
\newblock Counterfactual dependence and time?s arrow.
\newblock \emph{No\^{u}s}, 13\penalty0 (4):\penalty0 455--476, 1979{\natexlab{b}}.
\newblock \doi{10.2307/2215339}.

\bibitem[Lewis(1973{\natexlab{b}})]{Lewis1973-LEWC-2}
David~K. Lewis.
\newblock \emph{Counterfactuals}.
\newblock Blackwell, Malden, Mass., 1973{\natexlab{b}}.

\bibitem[Makinson and Schlechta(1994)]{Makinson94}
David Makinson and Karl Schlechta.
\newblock Local and global metrics for the semantics of counterfactual conditionals.
\newblock \emph{Journal of Applied Non-Classical Logics}, 4:\penalty0 129--140, 01 1994.
\newblock \doi{10.1080/11663081.1994.10510829}.

\bibitem[Patrick~Blackburn and Venema(2004)]{bluebook}
Maarten de~Rijke Patrick~Blackburn and Yde Venema.
\newblock Modal logic, cambridge tracts in theoretical computer science vol. 53.
\newblock \emph{Studia Logica}, 76\penalty0 (1):\penalty0 135--142, 2004.
\newblock \doi{10.1023/b:stud.0000027550.80518.77}.

\bibitem[Pollock(1976)]{Pollock1976-POLTPW}
John~L. Pollock.
\newblock The 'possible worlds' analysis of counterfactuals.
\newblock \emph{Philosophical Studies}, 29\penalty0 (6):\penalty0 469--476, 1976.
\newblock \doi{10.1007/bf00646329}.

\bibitem[Schikhof(1985)]{Schikhof_1985}
W.~H. Schikhof.
\newblock \emph{Ultrametric Calculus: An Introduction to p-Adic Analysis}.
\newblock Cambridge Studies in Advanced Mathematics. Cambridge University Press, 1985.

\bibitem[Stalnaker(1968)]{Stalnaker1968-STAATO-5}
Robert Stalnaker.
\newblock A theory of conditionals.
\newblock In Nicholas Rescher, editor, \emph{Studies in Logical Theory}, pages 98--112. Blackwell, 1968.

\bibitem[Woodward(2003)]{Woodward2003-WOOMTH}
James~F. Woodward.
\newblock \emph{Making Things Happen: A Theory of Causal Explanation}.
\newblock Oxford University Press, New York, 2003.

\end{thebibliography}
\end{document}